\documentclass[12pt,twoside,a4paper,reqno]{amspaper}
\usepackage{mathdef}
\usepackage{natbib}
\bibpunct{(}{)}{,}{a}{}{,}
\ifx\bibfont\newcommand{\bibfont}{\small}
\else\renewcommand{\bibfont}{\small}\fi

\usepackage{pst-all}
\usepackage{comment}
\usepackage{url}
\newcommand{\TN}{T}
\newcommand{\Th}{T_h}
\newcommand{\Thk}{T_{h,k}}

\newcommand{\g}{\phi}
\newcommand{\G}{\Phi}

\newcommand{\ghk}{\g^{h,k}}

\newcommand{\hatghk}{\hat\g^{h,k}}
\newcommand{\nuh}{\nu^h}
\newcommand{\hatnuh}{\hat\nu^{h,k}}

\newcommand{\Y}{Y^h}

\newcommand{\Yk}{Y^{h,k}}
\newcommand{\tildeY}{\tilde Y^h}
\newcommand{\tildeYk}{\tilde Y^{h,k}}

\newcommand{\barY}{\bar Y^h}
\newcommand{\barYk}{\bar Y^{h,k}}
\newcommand{\F}{\fclass}
\newcommand{\Gh}{\Gamma^h}

\newcommand{\Gk}{\Gamma_{h,k}}
\newcommand{\gk}{\gamma_{h,k}}
\newcommand{\hatGk}{\hat\Gamma_{h,k}}
\newcommand{\hatgk}{\hat\gamma_{h,k}}
\newcommand{\tildeGk}{\tilde\Gamma_{h,k}}
\newcommand{\tildegk}{\tilde\gamma_{h,k}}

\newcommand{\note}[1]
{$^{(!)}$\marginpar[{\hfill\tiny{\sf{#1}}}]{\tiny{\sf{(!) #1}}}}

\begin{document}
\title[Graphical Modeling for Multivariate Hawkes Processes]{Graphical Modeling for Multivariate Hawkes Processes with Nonparametric Link Functions}
\author{Michael Eichler, Rainer Dahlhaus, and Johannes Dueck}
\affil{Maastricht University and University of Heidelberg}
\dedicatory{{\upshape\today}}
\begin{abstract}
\citet{hawkes-71} introduced a powerful multivariate point process model of mutually exciting processes to explain causal structure in data. In this paper it is shown that the Granger causality structure of such processes is fully encoded in the corresponding link functions of the model. A new nonparametric
estimator of the link functions based on a time-discretized version of the point process is introduced by using an infinite order autoregression. Consistency of the new estimator is derived. The estimator is applied to simulated data and to neural spike train data from the spinal dorsal horn of a rat.\\[1em]
\smallskip

\end{abstract}
\maketitle

\section{Introduction}

In two seminal papers, \citet{hawkes-71,hawkes-71b} Hawkes introduced a multivariate model for point processes with mutually exciting components now referred to as the Hawkes model. In the beginning it was motivated by modeling aftershocks and seismological phenomena \citep[cf.][]{vere-70,vere-82,ogata-99}. It also served as a first model for neuron firing and stimulated the introduction of more complex nonlinear models to include inhibitory couplings and the refractory period \citep{okatan-05,cardanobile-10}. The usage of the Hawkes model has been more and more spread out to different research areas: \citet{brantingham-11} examines insurgency in Iraq, \citet{mohler-11} use it for modeling crime, \citet{reynaud-10} apply it to genome analysis, and \citet{carstensen-2010} model the occurrence of regulatory elements. Recently, the Hawkes model has become popular in particular in finance for modeling price fluctuations or transactions, cf. \citet{bacry-11}, \citet{bacry-12} and \citet{embrechts-11}.

In this paper we put the focus on the causal structure of the Hawkes model by applying causality concepts to mutually exciting point processes. \citet{granger-69} defined the notion of Granger causality. It reflects the belief that a cause should always occur before the effect and that the prediction of a process with the knowledge of a possible cause should improve if there is a causal relation present. However, temporal precedence alone is not a sufficient condition for establishing cause-effect relationships, and it is commonly accepted that empirical evaluation of Granger causality can lead to false detection of causal links. Nevertheless, the concept of Granger causality together with suitable graphical representations remains a useful tool for causal learning as has been shown in \citet{eichlerCSPA2012,eichlerPTRSA2013}.
The first objective of this paper is to set the framework for such causal learning approaches by defining the necessary graphical concepts. In particular, we establish a global Markov property, which relates the Granger causalities observed for part of the variables to the causal structure of the full system. Such global Markov properties play a key role in the graphical approach to causal learning.

The original definition of Granger applies only to processes in discrete time. Extensions to continuous time processes have been developed in the general framework of continuous-time semimartingales by \citet{florens-96} and for mean square continuous processes by \citet{comterenault96}. The notion of Granger causality in continuous time is closely related to the definition of local independence for composable Markov processes \citep{schweder70} and marked point processes \citep{Didelez-08}.

For multivariate point processes, Granger causal dependences can be described in terms of the conditional intensity $\lam(t)$, which in the case of a Hawkes process is of the form
\begin{align*}
\lambda(t)=\nu+\int_{0}^t \g(u)\,dN(t-u),
\end{align*}
where $\nu$ is a vector of positive constants (often referred to as background or Poisson rates) and $\g(\cdot)$ is a matrix of nonnegative link functions (also called Hawkes kernels) that vanish on the negative half axis. With this linear dependence structure, Hawkes processes may be viewed as a point process analogue to classical autoregressions in time series analysis. Whenever a component process produces an event, this increases the conditional firing rates of the other processes specified by the corresponding Hawkes link function $\phi(\cdot)$. Hence, the link functions encode a causal structure. The problem is to estimate these link functions. The most popular approach is a maximum likelihood approach as in \citet{ozaki-79}, where $\g(\cdot)$ is assumed to be of parametric form, e.g.~consisting of exponential functions or Laguerre polynomials. In recent literature, other estimation procedures have been proposed: \citet{bacry-12} use a numerical method for nonparametric estimation based on martingale and Laplace transform techniques, while the model-independent stochastic declustering (MISD) algortihm of \citet{marson2008,marson2010} employs the EM-algorithm to recover a piecewise constant approximation of the link function; the latter approach has been generalized by citet{lewis-11} to smooth function classes by using penalized maximization. As these methods involve complex computations, their use for causal learning algorithms, which require fitting of a large number of models, is limited.

In this paper, we present a simple and fast alternative to the existing nonparametric estimation method for the Hawkes link functions: we propose to discretize the point process by considering the the increments over equidistant time points and then to fit a vector autoregressive model by least squares. Part of the derivation will be along the ideas in \citet{lewis-85}. The paper is organized as follows: in Section \ref{section-hawkes} we define the Hawkes model and summarize some basic properties. Section \ref{section-granger} contains a discussion of the causality structure with respect to Granger causality and from the point of view of graphical models. In Section \ref{section-estimation} we introduce the nonparametric estimator and prove consistency when the observation interval tends to infinity and the discretization step size tends to zero. As an illustration, Section \ref{section-application} contains an application to EEG data from the spinal dorsal horn of a rat. Section \ref{section-conclusion} contains some concluding remarks. Part of the proofs have been put into the appendix. 

\section{Multivariate Hawkes processes}
\label{section-hawkes}

We consider multivariate point processes $N=(N_1,\ldots,N_d)'$ on a probability space $(\Omega,\fclass,\prob)$, that is, the components $N_i$, $i=1,\dots,d$ are random counting measures on $\rnum$. For simplicity, we will write $N_i(t)=N_i([0,t])$ for the number of events of the $i$--th component process up to time $t$. Throughout the paper, we make the following basic assumption.

\begin{assumption}
\label{assumppointproc}
The point process $N=(N_1,\ldots,N_d)'$ is stationary,
\[
N_1(A_1),\ldots,N_d(A_d)\sim N_1(t+A_1),\ldots,N_d(t+A_d)
\]
for all measurable sets $A_1,\ldots,A_d\subseteq\rnum$, where
$t+A=\{t+a\in\rnum|a\in A\}$. Furthermore, $N$ is a simple point process, that is, the counting processes $N_j$, $1\leq j\leq d$, have almost surely step size 1 and do not jump simultaneously.
\end{assumption}

Let $\fclass=(\fclass(t))_{t\in\rnum}$ be the filtration generated by the process $N$. Then $N$ is a special semimartingale with respect to $\fclass$, that is, there exists a predictable process $\Lambda=\big(\Lambda(t)\big)_{t\geq 0}$ with $\Lambda(0)=0$ such that the process $M=\big(M(t)\big)_{t\geq 0}$ given by the decomposition
\begin{equation}
\label{decomposition}
N(t)=M(t)+\Lambda(t)
\end{equation}
is a martingale with respect to $\fclass$. The process $\Lambda$ is called the compensator of $N$. The decomposition \eqref{decomposition} is important for the definition of Granger non--causality in section \ref{section-granger}.

Of particular interest is the case where the compensator process $\Lambda$ has components $\Lambda_i=\big(\Lambda_i(t)\big)_{t\geq 0}$ that are absolutely continuous with respect to Lebesgue measure.
\begin{definition}
Let $\lambda=\big(\lambda(t)\big)_{t\geq 0}$ be a $\fclass$-predictable process in $\rnum_+^d$ such that
\[
\int_0^\infty \lam_i(t)\,dt<\infty
\]
for every $i=1,\ldots,d$ and
\[
\mean\bigg(\int_0^\infty X(t)\,\lam(t)\,dt\bigg)
=\mean\bigg(\int_0^\infty X(t)\,dN(t)\bigg)
\]
for all non-negative predictable processes $X=(X(t))_{t\geq 0}$. Then the process $\lambda$ is the conditional intensity of the point process $N$.
\end{definition}
The conditional intensity $\lam(t)$ thus describes the intensity by which new events are generated conditionally on the past history of the process $N$,
\[
\lam(t)=\lim_{h\to 0}\tfrac{1}{h}\,\prob\big(N([t,t+h))\geq 1|\fclass(t-)\big),
\]
where $\fclass(t-)$ is the $\sigma$--algebra of all events up to but excluding time $t$.

In this paper, we are considering point processes for which the conditional intensity of each component depends linearly on the past events.
\begin{definition}
\label{hawkes-def}
A multivariate Hawkes process is a stationary and simple multivariate point process $N=(N_1,\ldots,N_d)'$ such that $N_i$ has conditional intensity
\begin{equation}
\label{Hawkes-intensity}
\lam_i(t)=\nu_i+\lsum_{j=1}^{d}\int_{0}^{t}\g_{ij}(u)\,dN_j(t-u),
\end{equation}
where $\nu_i>0$ is the baseline intensity of the $i$--th component and the link functions $\g_{ij}$ satisfy $\g_{ij}(u)=0$ for $u\leq 0$ and
\begin{equation}
\label{integrability}
\int_{0}^{\infty}\norm{\g_{ij}(u)}\,du<1.
\end{equation}
\end{definition}
The integrability condition \eqref{integrability} ensures the existence and uniqueness of a stationary point process with conditional intensities given by \eqref{Hawkes-intensity} \citep[e.g.][]{hawkesoakes1974}. Furthermore, the condition $\g_{ij}(0)=0$ ensures that the conditional intensity $\lam(t)$ is left-continuous and hence is predictable as required by the decomposition in \ref{decomposition}.

For a stationary Hawkes process, the mean intensity $p_N=\mean(N(1))$ is related to the link function $\g$ and baseline intensity $\nu$ by
\begin{equation}
p_N=\Big(I_d-\int_{0}^\infty \g(u)\,du\Big)^{-1}\,\nu
\label{intensity}
\end{equation}
where $I_d$ denotes the $d\times d$ identity matrix. The covariance structure of a stationary simple point process $N$ is given by
\[
\cov\big(N(A),N(B)\big)=\int_A\int_B q_{NN}(t-s)\,dt\,ds+\int_{A\cap B}p_N\,dt,
\]
were $q_{NN}$ is the covariance density of $N$. Here, the second integral is due to the fact that for a simple point process we have $\mean(dN(t)^2)=\mean(dN(t))$. For Hawkes processes, explicit expressions for the covariance density are available, for instance, in the case of exponentially decaying link functions \citep{hawkes-71}; \citet{bacry-12} provides a detailed analysis of the covariance structure in the general case.

Similarly, for integrable functions $f_1,\ldots,f_k$,
we have for the cumulants of higher order
\begin{align}
\label{cumulants}
\int_{\rnum^k}\lprod_{i=1}^k &f_i(t_i)\,
\cum\big(dN_{i_1}(t_1),\ldots,dN_{i_k}(t_k)\big)\notag\\
&=\sum_{P_1,\ldots,P_m}\int_{\rnum^m}
\prod_{j\in P_1}f_j(\tau_1)\,\cdots
\prod_{j\in P_m}f_j(\tau_m)\,q_{p_{11},\ldots,p_{1m}}(\tau_1,\ldots,\tau_m)
\,d\tau_1\cdots d\tau_m,
\end{align}
where the first sum extends over all partitions $\{P_1,\ldots,P_m\}$ of $\{1,\ldots,k\}$ with $m=1,\ldots,k$. Furthermore, $q_{i_1,\ldots,i_m}$ denotes the cumulant density of $N_{i_1},\ldots,N_{i_m}$ and the first sum extends over all partitions $\{P_1,\ldots,P_m\}$, $m=1,\ldots,k$, of $\{1,\ldots,k\}$. Explicit expressions for the cumulants of a Hawkes Process can be found in \citet{jovanovic2015}.

\section{Graphical modelling of multivariate Hawkes processes}
\label{section-granger}

Let $N=(N_1,\ldots,N_d)'$ be a stationary $d$-dimensional point process with canonical filtration $\fclass=\big(\fclass(t)\big)_{t\in\rnum}$. For any $A\subseteq V=\{1,\ldots,d\}$ let $\fclass_A(t)$ be the sub-$\sigma$-algebra that corresponds to the sub-process $N_A=(N_a)_{a\in A}$. Following the general definition of \citet{florens-96}, we obtain the following definition of Granger non-causality. We note that \citet{florens-96} use the term ``instantaneous Granger non-causality'' to distinguish it from non-causality over longer time horizons.

\begin{definition}[Granger non-causality]
Let $N$ be a stationary multivariate point process with canonical filtration $\fclass$. Then the $i$-th component $N_i$ does not Granger-cause the $j$-th component $N_j$ with respect to $\fclass$ if the compensator $\Lambda_j(t)$---or equivalently the conditional intensity function $\lambda_j(t)$---is $\fclass_{-i}(t-)$-measurable for all $t\in\rnum$.
\end{definition}

The above definition could be generalized by including also additional exogenous variables $X$ into the filtration. For our discussion of multivariate Hawkes processes as given by \ref{hawkes-def} it is sufficient to consider canonical filtrations generated by the point processes $N$.

For the $i$-th component $N_i$ of a Hawkes process, the conditional intensity $\lambda_i(t)$ is $\fclass_{-j}(t-)$-measurable if and only if the $j$-th link function $g_{ij}$ is identical to zero. Thus we have the straightforward result.

\begin{proposition}
Let $N$ be a multivariate Hawkes process with intensities as in \eqref{Hawkes-intensity}. Then $N_{i}$ does not Granger-cause $N_{j}$ with respect to $N$ if and only if $\g_{ji}(u)=0$ for all $u\in\rnum$.
\end{proposition}

The martingale property of the process $M$ in \eqref{decomposition} with respect to the filtration $\fclass$, that is, $\mean\big(dM(t)\given\fclass(t)\big)=0$, suggests that the above definition of Granger causality considers only dependence in the mean and hence corresponds to what is known as Granger causality in the mean in the context of time series. However, we note that for simple point processes the conditional intensities determine the full conditional distribution. This rules out any higher-order dependences of $M(t)$ on the past. Additionally, the increments of the components of $M(t)$ are mutually independent since simultaneous occurrence of events is almost surely not possible for a simple process. Consequently, the above definition in fact describes a notion of strong Granger (non-)causality formulated in terms of conditional independence.

\citet{schweder70} introduced the concept of local independence to describe dynamic dependences in time-continuous Markov processes. We note that the above notion of Granger noncausality and that of local independence are equivalent in the present context of point processes.

With the above definition of Granger noncausality, the definition of Granger causality graphs in \citet{eichlerpathdiagr,eichlergraphmodel} directly extends to the present case of multivariate Hawkes processes.

\begin{definition}
For a multivariate Hawkes process $N=(N_1,\ldots,N_d)'$, the Granger causality graph of $N$ is given by a graph $G$ with vertices $V=\{1,\ldots,d\}$ and directed edges $i\DE j$ with $i,j\in V$ satisfying
\[
i\DE j\notin G \iff \g_{ji}(u)=0\text{ for all }u\in\rnum.
\]
\end{definition}

The use of Granger causality graphs goes beyond simple visualization of dynamic dependences. The key feature of the graphical approach is that it relates Granger noncausality to pathwise separation in the graph by so-called global Markov properties. This allows to derive Granger noncausality relations for arbitrary subprocesses from the graph. More importantly, the graphical approach yields also criteria for identifying the Granger causal structure of a system which is only partially observed. This is particularly of interest in neurological applications where the activity of only a small number of neurons can be recorded.

The global Granger causal Markov property is a general result that goes beyond the framework of multivariate Hawkes processes. For multivariate simple stationary point processes, it simplifies compared to the time series case as the increments of the  martingale $M(t)$ are mutually independent. For its formulation, we need some terminology from graph theory.

Let $G$ be a directed graph with vertex set $V$ and edges $E$. A path in the graph $G$ is a sequence of edges $\pi=(e_1,\ldots,e_n)$ with $e_i\in\{a_{i-1}\rightarrow a_{i},a_{i-1}\leftarrow a_{i}\}$ for vertices $a_0,\ldots,a_n\in V$ with $a_0=a$ and $a_n=b$. If the last edge on the path is $a_{n-1}\rightarrow b$, that is, the path ends with an arrowhead at $b$, we speak of a path $\pi$ from $a$ to $b$ \citep[also referred to as $B$-pointing path, cf][]{eichlerpathdiagr}. A vertex $a_i$ on a path $\pi$ is called a collider if the adjacent edges form the subpath $a_{i-1}\rightarrow a_{i}\leftarrow a_{i+1}$; otherwise $a_{i}$ is called a non-collider. A path $\pi$ is blocked by a set $C$ if and only if there exists one collider on the path that does not lie in $C$ or there exists one non-collider that lies in $C$.

\begin{definition}[Global Markov properties]
A multivariate stationary simple point process $N$ satisfies the global Granger causal Markov property with respect to a directed graph $G$ if the following condition holds:
$N_A$ does not Granger cause $N_B$ with respect to $N_S$ if every path from a vertex $a\in A$ to a vertex $b\in B$ is blocked by the set $S\without A$.

Furthermore, we say that $N$ satisfies the global Markov property with respect to an undirected graph $U$ if the processes $N_A$ and $N_B$ are independent conditionally on $N_C$ whenever the sets $A$ and $B$ are separated by $C$ in $U$, that is, every path between some vertex $a\in A$ and some vertex $b\in B$ contains at least one vertex $c\in C$.
\end{definition}

The following result states that a stationary Hawkes processes $N$ is Granger-Markov with respect to its Granger-causality graph $G$. Likewise, $N$ is also Markov with respect to the moral graph $G\moral$ derived from the Granger-causality graph $G$. Here, the moral graph of a directed graph $G$ is defined as the undirected graph $G\moral$ that has the same vertex set as $G$ and has edges $i\UE j\in G\moral$ if $i$ and $j$ are adjacent in $G$ or there exists a third vertex $k$ such that $G$ contains both edges $i\DE k$ and $j\DE k$. Furthermore, a vertex $a$ is an ancestor of another vertex $b$ if there exists a path $a\DE\ldots\DE b$ in $G$; the set of all ancestors of vertices in 
$B\subseteq V$ is denoted by $\ancestor{B}$. Finally, $G_A$ for some subset $A\subseteq V$ denotes the subgraph obtained from the graph $G$ by retaining all edges that connect vertices in $A$.

\begin{theorem}
\label{globalMP}
Let $N$ be a stationary multivariate Hawkes process and let $G$ be the Granger causality graph of $N$. Then
\begin{romanlist}
\item
$N$ satisfies the global Granger causal Markov property with respect to $G$;
\item
every subprocess $N_S$, $S\subseteq V$, satisfies the global Markov property with respect to $(G_{\ancestor{S}\cup S})\moral$.
\end{romanlist}
\end{theorem}

We note that the graph $H=(G_{\ancestor{S}\cup S})\moral$ in (ii) can be reduced further to a graph $H(S)$ with vertex set $S$ by extending the subgraph $H_S$ by additional edges $i\UE j$ whenever $i$ and $j$ are not separated by $S\without\{i,j\}$ in $H$.
 
The global Granger causal Markov property allows an intuitive interpretation of pathways in Granger causality graphs. Moreover, it is of fundamental importance for graphical approaches to causal discovery. Here, the main problem is to distinguish true cause-effect relationships from so-called spurious causation due to unobserved variables. Under the global Granger-causal Markov property, the causal structure of the system including any relevant unobserved variables implies certain Granger non-causal relations among the observed variables and any subset thereof. Algorithms for causal discovery exploit this link by identifying all causal structures that are consistent with the observed Granger non-causal relations. For more details, we refer to \citet{eichlerCSPA2012,eichlerPTRSA2013}.

The largest problem for the implementation of such algorithms for causal discovery is the extremely large number of models that need to be fitted: $2^d-d-1$ models for $d$ variables. This prohibits the use of iterative methods for parameter estimation such as, for instance, the EM algorithm by \citet{lewis-11}. In the next section, we therefore discuss nonparametric estimation of the link function by discretizing the point process and applying standard least squares estimation for autoregressive time series.


\section{Nonparametric Estimation and Identification}
\label{section-estimation}

Our approach for nonparametric estimation of the link function $\g$ is via discretization and consequently using methods from time series analysis. Again, as in section \ref{section-hawkes} we observe a multivariate point process $N=\big(N(t)\big)_{t\in\rnum}$ with component processes $N_i$, $1\leq i \leq d$. The conditional intensity function is once more given by \eqref{Hawkes-intensity}, where the component functions of $\g$ belong to a class of non-negative integrable functions that is specified later on in the section. Our objective is the nonparametric estimation of the link functions of the Hawkes process, that is, we do not assume any parametric form of the link functions such as an exponential form. For the purpose of discretization we define for fixed $h>0$
\begin{align*} 
\Y_{i,t} = N_i(t\,h)-N_i((t-1)\,h)
\end{align*}
for all $t\in\znum$ and $1\leq i\leq d$. This is equivalent to dividing the real line into intervals of width $h$. For every fixed $h$, $\Y$ represents a $d$-dimensional time series displaying thge number of jumps in time intervals $((t-1)\,h,t\,h]$ for $t\in\znum$. Thus, for $h$ small enough, the random variables $\Y_t$ are approximately binary. Furthermore the considered point processes are orderly and hence the probability that more than one jump takes place in an interval of length $h$ is of order $o(h)$. This allows us to approximate the conditional mean $\mean[\Y_{i,t+1}|\F_{ht}]$ by
\begin{align*}
\mean[\Y_{i,t+1}|\F_{ht}]
&= \prob(N_i((t+1)\,h)-N_i(t\,h)=1\given\F_{ht})+o(h)\\
&= h\,\nu_i+h\,\lsum_{j=1}^d \int_{0}^{\infty} \g_{ij}(s)\,dN_j(t\,h-s) +o(h) \\
&= h\,\nu_i+h\,\lsum_{j=1}^d\lsum_{u=1}^{\infty} \int_{0}^{h} \g_{ij}(uh+\alpha)\,dN_j(t\,(h-u)-\alpha) + o(h).
\end{align*}
If the link function $\g$ is continuous and $h$ is small enough, we can approximate $\g$ by a piecewise constant function, which yields
\begin{equation}
\label{approximateARequation}
\mean(\Y_{t}|\F_{h(t-1)})\approx h\,\nu+h\,\lsum_{u=1}^{\infty} \g(u\,h)\,\Y_{t-u}.
\end{equation}
This suggests to estimate the link function by a least squares approach.

We introduce some notation. First let $\Gh(u)=\cov(\Y_{t},\Y_{t-u})$ for $u\geq 0$ and $\Gh(j)=\Gh(-j)'$ for $j<0$ be the covariance function of the process $\Y$, which depends on $h$. Furthermore with $\Yk_t=\vecc\big(\Y_{t-1},\ldots,\Y_{t-k}\big)$ we set 
\begin{align*}
\gk&=\cov(\Y_t,\Yk_t)=(\Gh(u))_{u=1,\ldots,k},\\
\Gk&=\cov(\Yk_t,\Yk_t)=(\Gh(u-v))_{u,v=1,\ldots,k}.
\end{align*}
Now suppose that the process $N$ has been observed on the interval $[0,\TN]$ and set $\Th=\TN/h$. Then the above linear approximation for the conditional mean of $\Y_t$ leads to the least squares problem of minimizing
\[
\lsum_{t=k+1}^{\Th}
\norm{\Y_t-\nuh-{\ghk}\,\Yk_t}^2_2.
\]
over the parameters $\nuh=\nu\,h$ and
\[
\ghk=(h\,\g(h),\ldots,h\,\g(hk)).
\]
The above expression is minimized by
\begin{align*}
\hatghk&=\hatgk\,\hatGk^{-1}\\
\intertext{and}
\hatnuh&=\barY-\hatghk\,\barYk,
\end{align*}
where with $\Thk=\Th-k$
\[
\hatgk=\tfrac{1}{\Thk}\lsum_{t=k+1}^{\Th} (\Y_t-\barY)(\Yk_t-\barYk)'
\]
is the sample covariance of $\Y_t$ and $\Yk_t$ and $\hatGk$, $\barY$, $\barYk$ are defined similarly.

With these definitions we are able to derive the desired asymptotic results for $\ghk$. Therefore we denote by $\norm{B}^2_2=tr(B'B)$ the Euclidean norm and by  $\norm{B}=\sup_{\norm{x}_2\leq 1}\norm{Bx}_2$ the spectral norm of $B$. We note that $\norm{B}^2$ equals the largest eigenvalue of the matrix $B'B$. For subsequent proofs recall the inequalities
$\norm{AB}_2\leq\norm{A}_2\,\norm{B}_2$ and 
$\norm{A}\leq\norm{A}_2\leq\sqrt{r}\norm{A}$, where $r$ is the rank of $A$. 

\begin{theorem}\label{convergence}
Let $N$ be a Hawkes process with baseline intensity $\nu$ and link function $\g$ satisfying Assumption \ref{assumppointproc}. Additionally suppose that the following conditions hold:
\begin{romanlist}
\item\label{A1}
Let $k=k_T$ and $h=h_T$ be functions of $T$ such that
\[
k_T\,h_T\to\infty,\qquad
k_T\,h_T^{2}\to 0,\qquad\text{and}\qquad
\SSS{\frac{k^2_T}{T}}\to 0\qquad\text{as }T\to\infty.
\]
as $T\to\infty$.
\item\label{A2}
The link function $\g$ satisfies $\bignorm{\int_0^\infty\g(u)\,du}<\infty$.
\item\label{A3}
The link function $\g$ is Lipschitz continuous and decreases to zero with $\norm{\g(u)}\leq C\,u^{-1}$ and 
\[
\int_{h_Tk_T}^{\infty}\norm{\g(v)}\,dv=o(1),\qquad T\to\infty.
\]
\end{romanlist} 
Then the least squares estimators $\hatghk$ and $\hatnuh$ are consistent,
\[
\norm{\hatghk-\ghk}_2\pconv 0
\qquad\text{and}\qquad
\norm{\hatnuh-\nuh}_2\pconv 0.
\]
as $T\to\infty$.
\end{theorem}
		
The first assumption of the above theorem requires usual rate conditions on the sequences $k_T$ and $h_T$. The first condition ensures that the support of the estimated link function increases with $T$ while the other two conditions restrict the growth of the number of parameters. The second assumption $\|\int_0^\infty \g(u) du\|<1$ ensures that the Hawkes process is stationary with absolutely integrable autocovariances density. Assumption (iii) restricts the tail of the link function; it is satisfied, for instance, for exponentially decreasing link functions and hence is not restrictive for practical applications.

The next theorem generalizes Theorem \ref{convergence} to a functional convergence. Therefore we set $\hat\g_T$ to be the step function defined by
\begin{align*}
\hat\g_T(u)=\tfrac{1}{h}\,\hatghk_{[u/h]},\qquad 0\leq u\leq k\,h
\end{align*}
and zero otherwise.

\begin{theorem}
Under the assumptions of Theorem \ref{convergence} it holds
\begin{align*}
\int_{0}^{\infty}\bignorm{\hat\g_T(u)-\g(u)}_2\,du\pconv 0\qquad\text{as } T\rightarrow\infty.
\end{align*}
\end{theorem}

\begin{proof}   
Decomposing the integral into the approximation error and the estimation error we find
\[
\int_{0}^{\infty} \norm{\hat\g_T(u)-\g(u)}_2\,du
\leq\int_{0}^{hk} \norm{\hat\g_T(u)-\g(u)}_2\,du
+\int_{hk}^{\infty}\norm{\g(u)}_2\,du.
\]
Here the second term is of order $O(h\sqrt{k})$ while the first term can be bounded by
\[
\norm{\hatghk-\ghk}_2+h\,\int_{0}^{hk}\norm{\g(uh)-\g([u+1]h)}_2\,du.
\]
Using Lipschitz continuity of the link function, the second term is of order $O(h^2\,k)=o(1)$ while the first term converges to zero in probability by Theorem \ref{convergence}.
\end{proof}

\begin{figure}
\centering
\includegraphics[width=11cm]{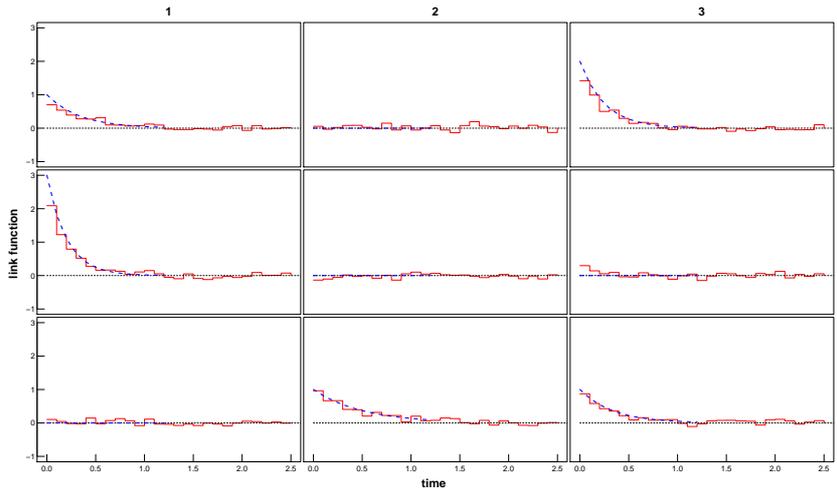}
\caption{Estimation of the Hawkes link functions for simulated data: nonparametric estimates (solid) and true link functions (dashed).}
\label{figuree} 
\end{figure}

Figure \ref{figuree} illustrates our estimation procedure based on simulated data. For the simulation of a three dimensional Hawkes process we used the method by \citet{ogata-81} based on a thinning algorithm. The true link functions, given by the dashed lines, were taken to have the form $\g_{ij}(u)=\alpha_{ij}\,\exp(-\beta_{ij}\,u)$ for $1\leq i,j\leq 3$ with different coefficients. In total, the simulated data contained approximately 7500 events. For the estimation of the link function, $h=0.1$ and $k=25$ were chosen as discretization parameters. The step functions in Figure \ref{figuree} are the obtained estimates of the link functions.


\subsection{Application}
\label{section-application}

As an application we analyzed neural spike train data from the lumbar spinal
dorsal horn of a pentobarbital-anaesthetised rat during noxious
stimulation. The firing times of ten neurons were recorded
simultaneously by a single electrode with an observation time
of 100s. The data have been measured and analyzed by \citet{sandk-94} who studied discharge patterns of spinal dorsal horn neurons under various conditions.

\begin{figure}\centering
\includegraphics[width=15cm]{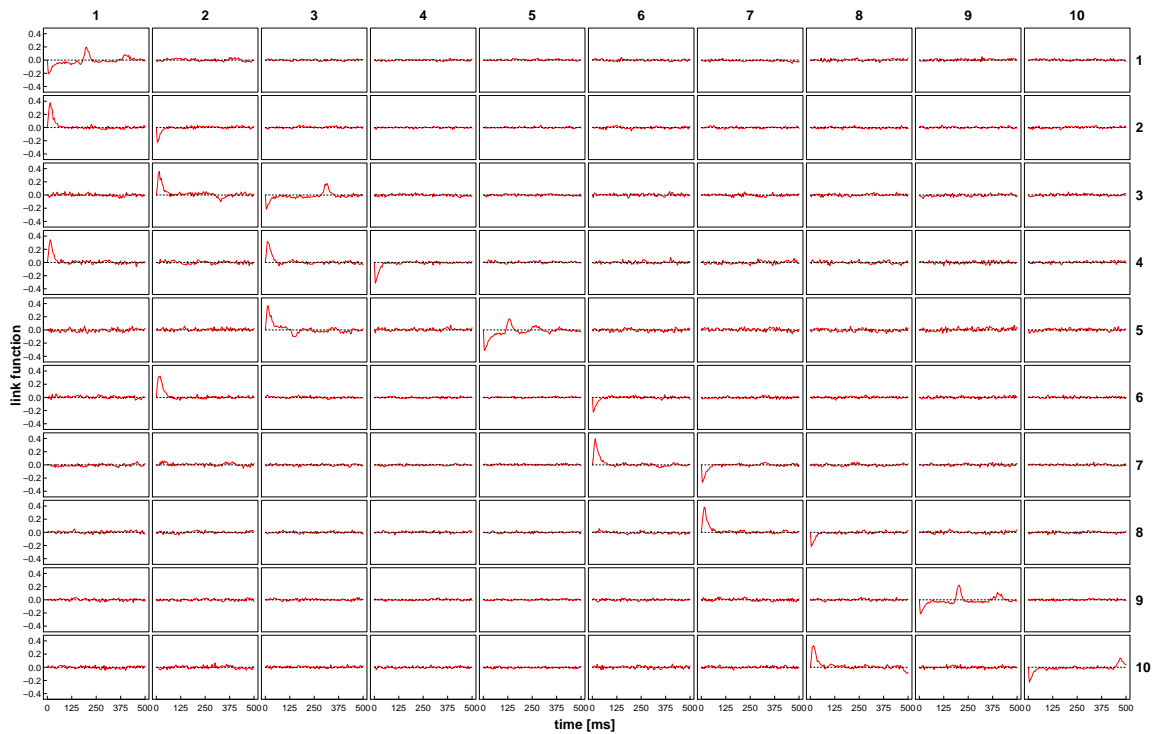}
\caption{Nonparametric estimates of the Hawkes link functions for the neural spike train data.}
\label{data}
\end{figure}

To determine the Granger causality graph for the ten neurons, we fit a simple multivariate Hawkes model to the data. As the firing pattern of neurons usually 
show a refractory period during which a neuron is less likely to fire again, a linear Hawkes model with non-negative link functions is not fully appropriate. However, our application shows that our estimation method is robust against such misspecifications. In contrast, more appropriate generalised Hawkes models \citep[e.g.][]{linigerphd} require optimisation techniques that are too time consuming to allow the analysis for a large number of subprocesses as required for causal learning. 

For the analysis we have chosen $h=0.5$ for the discretization parameter and $k=100$ for the order of the fitted time series model. The resulting nonparametric estimates of the link functions are shown in Figure \ref{data}. Additionally the estimated baseline intensities for the ten neurons are given in Table \ref{nutable}. As expected, the link functions on the diagonal indicate a self-inhibition after the firing of a neuron. In addition, five of the ten neurons show also an excitatory effect after time periods ranging from 125\,ms for neuron 5 to almost 500\,ms for neuron 10. This observation is in line with the rhythmic activity that is detectable in the spectra of the processes \citep{sandk-94}.

\begin{table}\centering
\caption{Estimates of the baseline intensities for neural spike train data.}
\footnotesize
\begin{tabular}{cccccccccc}
\hline
\hline
$\hat\nu_1$&$\hat\nu_2$&$\hat\nu_3$&$\hat\nu_4$&$\hat\nu_5$&
$\hat\nu_6$&$\hat\nu_7$&$\hat\nu_8$&$\hat\nu_9$&$\hat\nu_{10}$\\
\hline
0.1055 &0.0299 &0.1064 &0.0386 &0.1993 &0.0322 &0.0327 &0.0218 &0.1066 &0.0968\\
\hline
\hline
\end{tabular}
\label{nutable}
\end{table}

Next, we note that for most of the 90 possible directed links between the ten neurons the link function is approximately zero while only 10\% of the link functions show a clear positive peak indicating an excitatory effect. In all these cases the shape, time and intensity of the link functions are very similar with a peak at about 17\,ms and an intensity of approximately $0.38$ spikes per millisecond. For the nine clearly non-zero link functions $\g_{ij}$ we draw a corresponding edge $j\DE i$ in the Granger causality graph (Fig.~\ref{graph}). For instance, the link function $\g_{21}$ is non-zero and hence the graph contains an edge from node 1 to node 2. We note that in the graph, neuron nine is completely isolated from the other neurons. This is remarkable given the fact that it fires with the same frequency as for instance the first neuron. 
Finally we note that the decision whether a particular link function is statistically significantly non-zero should be based on a test; the construction of such a test that is feasible for a large number of link-functions and sub-models is planned in future work.

\begin{figure}\centering
\includegraphics[width=15cm]{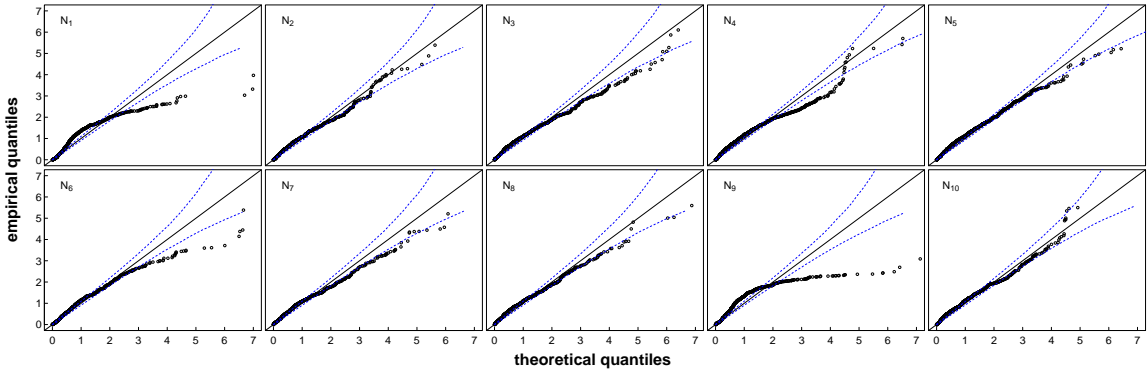}
\caption{Goodness-of-fit of Hawkes model for neural spike train data: quantile plot for the residual interarrival times; the dashed lines give the pointwise 95\% intervals for the interarrival times of a Poisson process.}
\label{residuals}
\end{figure}

For assessing the goodness-of-fit of the Hawkes model with nonparametrically estimated link functions, we consider the residual processes $R_i$, which are obtained from the event times $\tau_{ij}$, $j\in\nnum$, of the process $N_i$ by the random time change
\[
\sigma_{ij}=\Lambda_i(\tau_{ij}),
\]
that is, $R_i$ is the point process of events at times $\sigma_{ij}$, $j\in\nnum$. Then $R_i$, $i=1,\ldots,d$, are independent Poisson processes with unit intensity \citep[Thm 7.4.I]{daley-03}. Figure \ref{residuals} compares the quantiles of the interarrival times of the empirical residual process with the theoretical quantiles. With the exception of neurons 1 and 9, the model seems to fit the data reasonably well given the misspecification due to the refractory period. For neurons 1 and 9, the model cannot explain the strong rhythmic firing pattern visible in the data. This suggests that the rhythmic firing cannot be explained by a self-excitatory mechanism but is more likely caused by some external source. This might be captured by a time-varying baseline intensity with some seasonal pattern.

\begin{figure}\centering
\small
\newcommand{\myNode}[2]{\circlenode{#1}{\makebox[0.5ex]{#2}}}
\begin{pspicture}(0,0.2)(8.5,4.4)
  \psset{linewidth=0.6pt,arrowscale=1.5 2.0,arrowinset=0.2,dotsep=2pt}  
  \rput(0.5,1.5){\myNode{1}{$1$}}
  \rput(2.0,2.3){\myNode{2}{$2$}}
  \rput(3.5,1.5){\myNode{3}{$3$}}
  \rput(5.0,0.7){\myNode{4}{$4$}}
  \rput(5.0,2.3){\myNode{5}{$5$}}
  \rput(3.5,3.1){\myNode{6}{$6$}}
  \rput(5.0,3.9){\myNode{7}{$7$}}
  \rput(6.5,3.9){\myNode{8}{$8$}}
  \rput(6.5,1.5){\myNode{9}{$9$}}
  \rput(8.0,3.9){\myNode{10}{$10$}}
  \ncline{->}{1}{2}
  \ncline{->}{1}{4}
  \ncline{->}{2}{3}
  \ncline{->}{2}{6}
  \ncline{->}{3}{5}
  \ncline{->}{3}{4}
  \ncline{->}{6}{7}
  \ncline{->}{7}{8}
  \ncline{->}{8}{10}
\end{pspicture}
\caption{Granger causality graph obtained from the estimated link functions in Fig.~\ref{data} for the neural spike train data.}
\label{graph}
\end{figure}
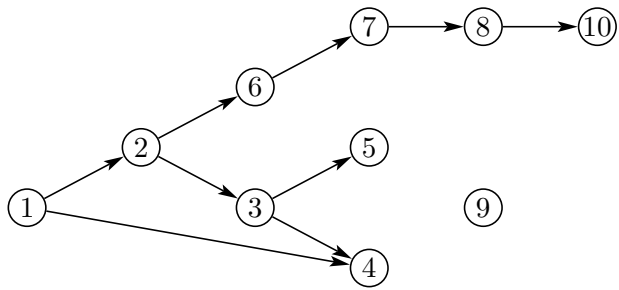

As already mentioned, the application of the Hawkes model to neural data suffers from two drawbacks. Firstly the model is only capable of modeling excitatory connections but not inhibition, which is well known to play a major role in neuronal connectivity. Secondly neurons possess a refractory period during which the neuron cannot fire again. Although these findings contradict the conditions of the Hawkes model, the above application demonstrates that the nonparametric estimates still yield meaningful results, whereas incorporating refractory periods into the model would destroy its linear structure. We note that in many other applications such as modeling aftershock effects \citep{ogata-99,vere-70,vere-82}, insurgency in Iraq \citep{brantingham-11}, crime \citep{mohler-11} or genome analysis \citep{reynaud-10} the Hawkes model does not encounter these problems or only to a lesser extent \citep{kagan2004}.

\section{Concluding Remarks}
\label{section-conclusion}

In this paper we have investigated the structure of Hawkes models and proved that the Granger causality structure of the process is fully encoded in the corresponding kernels. Moreover we have defined a new nonparametric estimator of the Hawkes kernels based on a time-discretized version of the point process and an infinite order autoregression. The estimator is easy and fast to compute even for higher dimensions, which is of particular importance for the implementation of causal search algorithms that require fitting not only of the full model but also of many submodels. We note that the computation of the required covariances is linear in the length of the discretized process. This allows to choose a high resolution for the nonparametric estimator even for long observation periods. In particular, the fitting of models to subprocesses for causal learning only depends on the chosen order of the approximating model since the covariances need to be computed only once from the discretized process. 

Given the form of the conditional intensity, and in particular \eqref{approximateARequation}, the estimator is quite intuitive. However, on a second glance, it is surprising that the method really is consistent since in the limit the discretized time series consists mainly of zeros and some 1s. We have succeeded to establish consistency rigorously but failed up to now to prove asymptotic normality - although we are still convinced that asymptotic normality holds with a reasonably good rate. A closer inspection of the problems reveals that the structure of the discretized time series is quite different from usual infinite order AR-processes in that the innovation are (approximately) a heteroscedastic martingale difference sequence leading to severe technical problems. Furthermore, some terms in the calculations are of higher order as in the AR-case and do not disappear. For this reason we have postponed the proof of asymptotic normality to future work.

\begin{appendix}
\section{Proofs}

\begin{proof}[Proof of Theorem \ref{globalMP}]
We consider stationary multivariate point processes on $\rnum$ while the proof of Didelez only covers processes on $[0,T]$ (or any compact interval). Therefore, strictly speaking, the result must be extended.

The log-likelihood of the process $N$ on $[t_0,t]$ is given by
\[
\lsum_{i=1}^{d}\bigg[\int_{t_0}^{t}\log\lambda_i(t)\,dN_i(t)
-\int_{t_0}^{t}\lambda_i(t)\,dt\bigg]
\]
where $\lambda_i(t)$ is given by \eqref{Hawkes-intensity} and only depends on
$N_{\closure{i}}(s)$, $s\leq t$, where $\closure{A}=A\cup\parent{A}$. It follows that the likelihood can be factorized into factors that are $\fclass_{C}(t)$-measurable for 
sets $C\in\mathcal{C}=\{\closure{i}|i\in V\}$. The sets in $\mathcal{C}$ are complete in the moral graph ${G\moral}$; combining factors with sets in the same clique we obtain a factorization over the cliques of the moral graph. The factorization prevails if we let $t_0$ tend to $-\infty$. This implies the global Markov property with respect to the moral graph.

Finally we note that every path from $A$ to $B$ is blocked by $B\cup C$ if and only if
\[
\parent{B}\without(B\cup C)\sep A\given B\cup C
\]
in the moral graph $(G_\ancestor{A\cup B\cup C})\moral$. This implies by the global Markov property that
\[
\fclass_{\parent{B}}(t)\indep\fclass_{A}(t)\given\fclass_{B\cup C}(t).
\]
Finally we have for every $b\in B$
\begin{align*}
\lambda_{b}^{A\cup B\cup C}(t)
&=\mean\big(\lambda_{b}^{V}(t)\big|\fclass_{A\cup B\cup C}(t)\big)
=\mean\big(\lambda_{b}^{\parent{b}}(t)\big|\fclass_{A\cup B\cup C}(t)\big)\\
&=\mean\big(\lambda_{b}^{\parent{b}}(t)\big|\fclass_{B\cup C}(t)\big)
=\mean\big(\lambda_{b}^{V}(t)\big|\fclass_{B\cup C}(t)\big)
=\lambda_{b}^{B\cup C}(t).
\end{align*}
Hence $N_A$ does not Granger cause $N_B$ with respect to the subprocess $N_{A\cup B\cup C}$.
\end{proof}

For the proof of Theorem \ref{convergence}, we need the following two technical lemmas.

\begin{lemma}
\label{tildeapprox}
Under the assumptions of Theorem \ref{convergence} we have
\begin{romanlist}
\item
$\DS\Bignorm{\tfrac{1}{\Thk}\lsum_{t=k+1}^{\Th}\big(\Y_{t-u}-p_N\,h\big)}^2_2
=o_P\big(\tfrac{h^2}{T})$;
\item
$\DS\norm{\hatGk-\tildeGk}^2_2=O_P\big(\tfrac{k^2\,h^4}{T^2}\big)$
and 
$\DS\norm{\hatgk-\tildegk}^2_2=O_P\big(\tfrac{k\,h^4}{T^2}\big)$
\end{romanlist}
where $\tildeGk$ and $\tildegk$ are defined with $\barY$ and $\barYk$ substituted by their mean $p_N\,h$.  
\end{lemma}

\begin{proof}
For (i), we note that
\[
\lsum_{t=k+1}^{\Th}\big(\Y_{t-u}-p_N\,h\big)=\int_{kh}^{T}d\tilde N(t-hu).
\]
This implies that
\begin{align*}
\mean\Bignorm{\SSS{\frac{1}{\Thk}}&\lsum_{t=k+1}^{\Th}
\big(\Y_{t-u}-p_N\,h\big)}^2_2\\
&=\SSS{\frac{1}{\Thk^2}}\int_{hk}^{T}\int_{hk}^{T}
\big(q_{NN}(t-s)+p_N\,\delta(t-s)\big)\,dt\,ds=O\Big(\SSS{\frac{h}{\Thk}}\Big).
\end{align*}
For the first part of (ii), we note that
\[
\norm{\hatGk-\tildeGk}_2
=(\barYk-(p_N\,h)\otimes I_d)'(\barYk-(p_N\,h)\otimes I_d)
\]
and hence
\[\mean\norm{\hatGk-\tildeGk}_2
=\lsum_{u=1}^{k}\mean\Bignorm{\SSS{\frac{1}{\Thk}}
\lsum_{t=k+1}^{\Th}\big(\Y_{t-u}-p_N\,h\big)}^2_2
=O\Big(k\,\SSS{\tfrac{h}{\Thk}}\Big)
\]
by (i). The second part of (ii) follows similarly.
\end{proof}

\begin{lemma}
\label{cov-orders}
Under the assumptions of Theorem \ref{convergence} we have
\begin{romanlist}
\item
$\DS\norm{\Gk}=O(h)$;
\item
$\DS\norm{\Gk^{-1}}=O\big(h^{-1}\big)$;
\item
$\DS\bignorm{\hatGk-\Gk}^2_2=O_p\big(\tfrac{k^2\,h^2}{T}\big)$;
\item
$\DS\bignorm{\hatGk^{-1}-\Gk^{-1}}^2_2=O_p\big(\tfrac{k^2}{h^2\,T}\big)$;
\item
$\DS\bignorm{\hatGk^{-1}}=O_p\big(h^{-1}\big)$.
\end{romanlist}
\end{lemma}
\begin{proof}
For the first two assertions, we note that $\Y$ is a discretized version of a linear transform of the stationary point process $N$. Thus we obtain for the covariance function of $\Y$
\begin{align*}
\cov(\Y_{u+1},\Y_{1})
&=\int_0^h\int_0^h\big(q_{NN}(uh+t-s)+p_N\,\delta(uh+t-s)\big)\,dt\,ds\\
&=\SSS{\frac{1}{2\pi}}\int_\rnum\int_0^h\int_0^h
\hat q_{NN}(\omega)\,e^{\im\omega(t-s+uh)}\,d\omega\,dt\,ds
+p_N\,h\,\delta(u)\\
\intertext{and further with $H_h(\omega)=\int_0^h I_d\,e^{-\im\omega t}\,dt$}
&=\SSS{\frac{1}{2\pi}}\int_\rnum
H_h(-\omega)\,\hat q_{NN}(\omega)\,H_h(\omega)\,e^{\im\omega\,h\,u}\,d\omega
+p_N\,h\,\delta(u)\\
&=\SSS{\frac{1}{2\pi\,h}}\int_\rnum H_h\big(-\tfrac{\omega}{h}\big)\,\hat q_{NN}\big(\tfrac{\omega}{h}\big)\,H_h(\big(\tfrac{\omega}{h}\big)
\,e^{\im\omega\,u}\,d\omega
+p_N\,h\,\delta(u)\\
&=\SSS{\frac{h}{2\pi}}\int_{-\pi}^{\pi}\Big[\lsum_{u\in\znum} H_1(-\omega-u)\,\hat q_{NN}\Big(\SSS{\frac{\omega+u}{h}}\Big)\,H_1(\omega+u)
+\SSS{\frac{1}{2\pi}}\,p_N\Big]\,e^{\im\omega\,u}\,d\omega,
\end{align*}
where we have used that $H_h\big(\tfrac{\omega}{h}\big)=h\,H_1(\omega)$. 
which implies
\[
f_{\Y\Y}(\omega)=h\,\lsum_{u\in\znum} H_1(-\omega-u)\,\hat q_{NN}\Big(\SSS{\frac{\omega+u}{h}}\Big)\,H_1(\omega+u)
+\SSS{\frac{h}{2\pi}}\,p_N.
\]
Since each summand in the first term is positive definite, we find that
\[
f_{\Y\Y}(\omega)\geq \SSS{\frac{h}{2\pi}}\,\min\{p_{N,1},\ldots,p_{N,d}\}\,I_d
\]
for all $\omega\in[-\pi,\pi]$ and all $h>0$. This proves that 
\[
\bignorm{\Gk^{-1}}=O\big(h^{-1}\big)
\]
for $h\to 0$ (the bound does not depend on $k$). Furthermore, under the assumptions on the link function the spectrum of $N$,
\[
f_{NN}(\omega)=(I_d-\G(\omega))^{-1}\,D_N\,(I_d-\G(-\omega)')^{-1}
\]
and hence $\hat q_{NN}(\omega)=f_{NN}(\omega)-\tfrac{1}{2\pi}\,D_N$ is uniformly bounded for all $\omega\in\rnum$. Since $|H_1(\omega+u)|^2$ satisfies
\[
\lsum_{u\in\znum}|H_1(\omega+u)|^2\leq C
\]
for all $\omega\in[-\pi,\pi]$, we get $\norm{\Gk}_2\leq\norm{f_{\Y\Y}(\omega)}\leq C\,h$ for some constant $C>0$.

For the third assertion, we first use the triangle inequality to get
\[
\bignorm{\hatGk-\Gk}_2\leq\bignorm{\hatGk-\tildeGk}_2+\bignorm{\tildeGk-\Gk}_2.
\]
Since by Lemma \ref{tildeapprox} the first term has the required order, it suffices to prove the assertion for the second term. We note
that $\mean(\tildeGk)=\Gk$ and $\mean(\tildeY_t)=0$. Thus we obtain by the product formula for cumulants
\begin{align*}
\mean\big(\norm{\tildeGk-\Gk}^2_2\big)
&=\tfrac{1}{\Thk^2}\lsum_{i_1,\ldots,i_4=1}^{d}\lsum_{u_1,\ldots,u_4=1}^{k}
\lsum_{t,s=k+1}^{\Th}
\cum\big(\tildeY_{i_1,t-u_1}\tildeY_{i_2,t-u_2},\tildeY_{i_3,s-u_3}\tildeY_{i_4,s-u_4}\big)\\
&=\tfrac{1}{\Thk^2}\lsum_{i_1,\ldots,i_4=1}^{d}\lsum_{u_1,\ldots,u_4=1}^{k}
\lsum_{t,s=k+1}^{\Th}
\Big[\cum\big(\tildeY_{i_1,t-u_1},\tildeY_{i_2,t-u_2},\tildeY_{i_3,s-u_3},\tildeY_{i_4,s-u_4}\big)\\
&\qquad\qquad+\cum\big(\tildeY_{i_1,t-u_1},\tildeY_{i_3,s-u_3}\big)
\cum\big(\tildeY_{i_2,t-u_2},\tildeY_{i_4,s-u_4}\big)\\
&\qquad\qquad+\cum\big(\tildeY_{i_1,t-u_1},\tildeY_{i_4,s-u_4}\big)
\cum\big(\tildeY_{i_2,t-u_2},\tildeY_{i_3,s-u_3}\big)\Big].
\end{align*}
By \eqref{cumulants} we find that all cumulants summed over $t$ are at most of order 
$O(h)$ which yields the required order.

The fourth result can be derived from (iii) similarly as in the proof of Lemma 3 of \citet{berk-74} by noting that
\begin{align*}
\bignorm{\hatGk^{-1}-\Gk^{-1}}_2
&=\bignorm{\hatGk^{-1}\big(\hatGk-\Gk\big)\Gk^{-1}}_2\\
&\leq\Big(\bignorm{\hatGk^{-1}-\Gk^{-1}}+\bignorm{\Gk^{-1}}\Big)\,%
\bignorm{\Gk^{-1}}\,\bignorm{\hatGk-\Gk}_2.
\end{align*}
Rewriting this as
\[
\big(1-\bignorm{\Gk^{-1}}\,\bignorm{\hatGk-\Gk}\big)\,
\bignorm{\hatGk^{-1}-\Gk^{-1}}_2
\leq\bignorm{\Gk^{-1}}^2\,\bignorm{\hatGk-\Gk}_2,
\]
we obtain the required convergence since 
by (ii) and (iii) $\bignorm{\Gk^{-1}}\,\bignorm{\hatGk-\Gk}
=O_P\big(\sqrt{\tfrac{k^2}{T}}\big)=o_P(1)$ and $\bignorm{\Gk^{-1}}^2\,\bignorm{\hatGk-\Gk}
=O_P\big(\sqrt{\tfrac{k^2}{h^2\,T}}\big)$.

For (v), we finally note that
\[
\bignorm{\hatGk^{-1}}
\leq\bignorm{\Gk^{-1}}+\bignorm{\hatGk^{-1}-\Gk^{-1}}=O_p\big(h^{-1}\big)
\]
by (ii) and (iv).
\end{proof}

In the following, $A\otimes B$ will denote the Kronecker product of matrices $A$ and $B$ with suitable dimensions.

\begin{proof}[Proof of Theorem \ref{convergence}]
First of all we note that assumption \ref{A2} assures the stationarity of the process $N$ and hence of the processes $\Y$ for all $h>0$. For notational convenience, let $d\tilde N(t)=dN(t)-p_N\,dt$. We start by rewriting 
\begin{align*}
\hatghk-\ghk
&=\big(\hatgk-\ghk\,\hatGk\big){\hatGk}^{-1}\\
&=\big(\tildegk-\ghk\,\tildeGk\big){\hatGk}^{-1}
+\big(\hatgk-\tildegk){\hatGk}^{-1}+\ghk(\tildeGk-\hatGk){\hatGk}^{-1}.
\end{align*}
By Lemmas \ref{tildeapprox} and \ref{cov-orders} the last two terms are in Euclidean norm at most of order $O_P(kh/T)$ and thus converge to zero in probability. For the first term, we get
\[
\big(\hatgk-\ghk\,\hatGk\big){\hatGk}^{-1}=\tfrac{1}{\Thk}\lsum_{t=k+1}^{\Th}\veps^h_t\,(\tildeYk_{t})'\,{\hatGk}^{-1},
\]
where $\veps^{h,k}_t=\tildeY_t-\ghk\,\tildeYk_t$. This yields the upper bound
\begin{equation}
\label{U-terms}
\bignorm{\hatghk-\ghk}_2
\leq\bignorm{\hatGk^{-1}}_2
\big(\norm{U_{1,T}}+\norm{U_{2,T}}+\norm{U_{3,T}}\big)+o_P(1)
\end{equation}
with 
\begin{align*} 
U_{1,T} &=\SSS{\frac{1}{\Thk}}\lsum_{t=k+1}^{\Th}
\big[\tildeY_{t}-\mean\big(\tildeY_t|\F_{t\,h-h}\big)\big]\,(\tildeYk_{t})',\\
U_{2,T} &=\SSS{\frac{h}{\Thk}}\lsum_{t=k+1}^{\Th}
\Big(\int_0^{hk}\g(u)\,d\tilde N(t\,h-u)-\lsum_{u=1}^k \g(h\,u)\tildeY_{t-u}\Big)\,(\tildeYk_{t})',\\
U_{3,T} &=\SSS{\frac{h}{\Thk}}\lsum_{t=k+1}^{\Th}
\Big(\int_{hk}^\infty\g(u)\,d\tilde N(t\,h-u)\Big)\,(\tildeYk_{t})',
\end{align*}							 
where we have used that $p_N=\nu+\int \g(u)\,du\,p_N$. We proceed in showing that the three terms $U_{1,T}$, $U_{2,T}$ and $U_{3,T}$ are of order $o_P(h)$. Together with Lemma \ref{cov-orders} (ii) this proves $\norm{\hatghk-\ghk}_2=o_P(1)$. 

Starting with the first term $U_{1,T}$,  we find
\begin{align*}
\mean\norm{U_{1,T}}^2
&=\SSS{\frac{1}{\Thk^2}}\mean\Bignorm{\lsum_{t=k+1}^{\Th}
\big(\tildeY_t-\mean(\tildeY_t\given\F_{h(t-1)})\big)\,(\tildeYk_{t})'}^2\\ 
&=\SSS{\frac{1}{\Thk^2}}\sum_{t,s=k+1}^{\Th}\mean\Big[
\big(\tildeY_s-\mean(\tildeY_s\given\F_{h(s-1)})\big)'
\big(\tildeY_t-\mean(\tildeY_t\given\F_{h(t-1)})\big)
\,(\tildeYk_{t})'\,{\tildeYk_{s}}\Big]\\
\intertext{and further by the martingale property}
&=\SSS{\frac{1}{\Thk^2}}\sum_{t=k+1}^{\Th}
\mean\Big[\bignorm{\tildeY_t-\mean(\tildeY_t\given\F_{h(t-1)})}_2^2
\bignorm{\tildeYk_{t}}^2_2\Big]\leq\SSS{\frac{C\,h}{\Thk}}.
\end{align*}
Thus $h^{-1}\,\norm{U_{1,T}}$ converges to zero  in probability according to assumption (i).
	
For the second term in \eqref{U-terms} we obtain
\begin{align*}  
\mean\norm{U_{2,T}}
&\leq\SSS{\frac{h}{\Thk}}\lsum_{t=k+1}^{\Th}
\Big(\mean\Bignorm{\lsum_{u=1}^k\int_0^h\big(\g(hu+\alpha)-\g(hu)\big)
\,d\tilde N(t(h-u)-\alpha)}_2^2
\,\mean\bignorm{\tildeYk_t}_2^2\Big)^{\tfrac{1}{2}}.
\end{align*}
By Lemma \ref{cov-orders} we obtain for the second mean
$\mean\bignorm{\tildeYk_t}_2^2=\trace{\Gk}=O(h\,k)$. For the first mean, we have
\begin{align*}
\mean\Bignorm{\lsum_{u=1}^k&\int_0^h\big(\g(hu+\alpha)-\g(hu)\big)
\,d\tilde N(t(h-u)-\alpha)}_2^2\\
&=\lsum_{u,v=1}^{k}\int_0^h\int_0^h
\trace\big[\big(\g(hu+\alpha)-\g(hu)\big)\,q_{NN}(h(u-v)+\alpha-\beta)\\
&\qquad\qquad \times\big(\g(hv+\beta)-\g(hv)\big)'\big]\,d\alpha\,d\beta\\
&\qquad\qquad +\lsum_{u=1}^{k}\int_0^h
\trace\big[\big(\g(hu+\alpha)-\g(hu)\big)\,D_{N}\,\big(\g(hu+\alpha)-\g(hu)\big)'\big]\,d\alpha\,d\beta\\
&\leq C\,h^2\int_0^{hk}\int_0^{hk}\norm{q_{NN}(\alpha-\beta)}_1\,d\alpha\,d\beta
+C\,h^3\,k\,\norm{D_N}_1\leq C\,h^3\,k.
\end{align*}
Combining the results, we find $\mean\norm{U_{2,T}}_2=O(h^3\,k)=o(h)$.

For the last term in \eqref{U-terms} we obtain similarly as for $U_{2,T}$
\begin{align*}
\mean\Bignorm{\int_{hk}^\infty&\g(u)\,\big(dN(t\,h-u)-p_N\,du\big)}^2_2\\
&\leq\int_{kh}^{\infty}\int_{kh}^{\infty}\norm{\g(u)}\,\norm{\g(v)}
\,\norm{q_{NN}(u-v)}_2\,du\,dv
+\int_{kh}^{\infty}\norm{\g(u)}^2\,\norm{D_{N}}_2\,du=O\big(1/hk\big),
\end{align*}
where we have used assumption (iii). Together with $\mean\norm{\tildeYk_t}_2^2=O(h\,k)$ this yields $\mean\norm{U_{3,T}}_2=o(h)$.

Finally, we note for the consistency of $\hatnuh$ that
\begin{align*}
\hatnuh-\nuh
&=\SSS{\frac{1}{\Thk}}\lsum_{t=k+1}^{\Th}
\big(\Y_t-\nuh-\hatghk\Yk_t\big)\\
&=\SSS{\frac{1}{\Thk}}\lsum_{t=k+1}^{\Th}
\big(\tildeY_t-\mean(\tildeY_t\given\F_{h(t-1)})\big)
+\SSS{\frac{1}{\Thk}}\lsum_{t=k+1}^{\Th}
\big(\ghk-\hatghk)\tildeYk_t\\
&\qquad\qquad+\SSS{\frac{h}{\Thk}}\lsum_{t=k+1}^{\Th}
\lsum_{u=1}^{k}\int_0^h\big(\g(hu+\alpha)-\g(hu)\big)\,dN(t(h-u)-\alpha)\\
&\qquad\qquad+\SSS{\frac{h}{\Thk}}
\lsum_{t=k+1}^{\Th}\int_{hk}^{\infty}\g(u)\,dN(th-u)+o_P(h)
\end{align*}
Convergence to zero of all four terms follows by similar arguments as above.
\end{proof}

\end{appendix}
\bibliography{bib}
\bibliographystyle{stat}

\end{document}